\documentclass [12pt] {article}
\author{Matthieu Arfeux\\ {\small matthieu.arfeux@stonybrook.edu}}

\usepackage[latin1]{inputenc}%
\usepackage{amssymb,amsfonts,epsfig,amsmath,graphics,graphicx}%
\usepackage{amsthm}%
\usepackage[all]{xy}%
\usepackage{verbatim}%
\usepackage{enumitem}
\usepackage{bm}
\usepackage[nottoc, notlof, notlot]{tocbibind}

\title {Berkovich spaces and Deligne-Mumford compactification}

\newtheorem{theorem} {Theorem}[section]

\newtheorem{proposition}[theorem]{Proposition}
\newtheorem{lemma}[theorem]{Lemma}
\newtheorem*{lemma*}{Lemma}
\newtheorem{corollary}[theorem]{Corollary}
\newtheorem{definition}[theorem]{Definition}

\newtheorem*{conjecture*}{Conjecture}

\newtheorem{example}[theorem]{Example}

\theoremstyle{plain}

\newtheorem*{definitionint}{Definition}

\newtheorem{noteint}{Note}

\theoremstyle{remark}
\newtheorem{remark}[theorem]{Remark}

\theoremstyle{remark}

\renewenvironment{proof}{\noindent{\bf Proof. }}{\hfill{$\square$} \vskip.3cm}

\def\cal{\mathcal}

\def\C{{\mathbb C}}

\def\F{{\mathcal F}}

\def\LL{{\mathbb L}}
\def\N{{\mathbb N}}
\def\P{{\mathbb P}}

\def\Q{{\mathbb Q}}
\def\R{{{\mathbb R}}}
\def\S{{\mathbb S}}
\def\St{{\mathcal S}}
\def\T{{\mathcal T}}

\def\Z{{\mathbb Z}}

\def\Aut{{\rm Aut}}
\def\Rat{{\rm Rat}}
\def\rat{{\rm rat}}

\def\Mod{{\rm Mod}}

 \def\epsilon{{\varepsilon}}

\def\card{{\rm card}}
\def\deg{{\rm deg}}

\usepackage{hyperref}

\hypersetup{
backref=true, 
pagebackref=true,
hyperindex=true, 
colorlinks=true, 
breaklinks=true, 
urlcolor= blue, 
linkcolor= blue, 
}

\begin{document}

\maketitle

\begin{abstract}
We explicit the relation between the dynamics the Berkovich projective line over the completion of the field of formal Puiseux series and the space dynamical systems between trees of spheres known to be equivalent to the Deligne-Mumford compactification of the Moduli space of marked spheres.
\end{abstract}



\medskip
\noindent{\textbf{About this preprint.}}This preprint does not expose any new result. It may be completed later in a version that would contain a description of its relation between different other works and also concrete application to dynamics but the author wanted to make it quickly available. The readers are very welcome to mail any comment that could improve this preprint.

\section{Introduction}

\medskip
\noindent{\textbf{Motivations.}}

Who in holomorphic dynamics did not hear about Berkovich spaces and about Deligne-Mumford compactification? Indeed these tools appear more and more in the literature these last years, explicitly as in \cite{K1}, \cite{DF}, \cite{HK}, \cite{S}, \cite{A0} for example or implicitly as in DeMarco-McMullen's trees \cite{DM}, Shishikura's trees \cite{Sh1}, etc... It would really be a very long task to list where these tools are used! 
However, it is often painful for people to be introduced in these topics because the main ideas are often hidden under a heavy formalism that has to be introduced in order to state corrects results. 

What if I tell you that in fact these tools behind their respective formalisms are very simple and that they are also related?
One could think that explaining the relation would be very painful and that if it is already hard to learn about one of these topics, it would be crazy to try to understand both of them at the same time. But one would be wrong, as this is exactly the goal of this paper:
\begin{itemize}
\item introducing both of these topics at the same time,
\item explaining the main ideas behind them, 
\item expliciting the relation between them, and
\item try to be understandable by a beginner in both of these fields.
\end{itemize}

First let us recall what I mean by holomorphic dynamics.
Let us denote by $\S:=\P^1\C$ the Riemann sphere. According to the Uniformization Theorem, every compact surface of genus $0$ with a projective structure is isomorphic to $\S$. For $d\geq 1$, we denote by $\Rat_d$ the set of rational maps $f:{\mathbb S}\to {\mathbb S}$ of degree $d$. We want to understand the behavior of these maps under iteration. Hence we introduce the natural action by conjugacy of the set $\Aut(\S):=\Rat_1$ of Moebius transformations on $\Rat_d$ :
\[\Aut(\S)\times \Rat_d\ni (M,f)\mapsto M\circ f\circ M^{-1}\in \Rat_d.\]  
We want to understand the dynamical properties of $\rat_d$, the quotient of $\Rat_d$ by this action. 
Here we are going to focus on the study of diverging sequences of $\rat_d$.

Berkovich spaces and the Deligne-Mumford compactification are very general tools that can be used in very general settings but for our purpose we will restrict our use to the completion of the field of formal Puiseux series for Berkovich spaces and to the case of genus zero (moduli space of punctured spheres) for the Deligne-Mumford compactification. 

\begin{noteint}For the Berkovich space, we will recall what is this field later (as late as possible). This will restrict our application of Berkovich spaces in dynamics (for example we will not be able to explain how they are used in \cite{BD} or in \cite{FG}). 
\end{noteint}

\begin{noteint}For the Deligne-Mumford compactification, a new vocabulary with dynamical systems between trees of spheres, that was claimed to be more adapted for the use in holomorphic dynamics, has been introduced in \cite{A0} (translated in English in \cite{A1}, \cite{A2}, \cite{A3}). We will directly use it without going back to the formalism of Deligne-Mumford. 
\end{noteint}

In \cite{K2}, the author explain how to use Berkovich space ideas in order to deduce properties of diverging sequences in $\rat_d$ and the vocabulary in \cite{A0} has been based on this idea. One of the goal of this work is also to explicit here the bridge between the two formalism in order to allow later to use this bridge the reverse way. This paper is based on some remarks sketched in \cite{A0}.


\medskip
\noindent{\textbf{A problem of normalization.}}

As we said, the space we are interested in, $\rat_d$, is a quotient space and it is usually convenient to use representatives instead of classes. The choice of representatives (often called normalizations) is usually a problem when one want to study the behavior in the boundary of the space.
For example, it can happen that a sequence diverges in $\Rat_d$ and at the same time converges in $\rat_d$.

When we go to infinity we always have some critical points collapsing but the reverse can be false. Intuitively it is natural to think that one could understand the behavior of a diverging sequence by trying to understand how the critical points are collapsing together. That's why it is natural to keep trace of the critical points. Of course, as we are doing dynamics, we will also look at the behavior of periodic orbits.
To keep trace of these elements, we "mark" them on the sphere.
Given a rational map $f\in\Rat_d$ and a finite set $X$ of elements which are meaningful for us (critical points, orbits,...), we define a sphere marked by $X$ to be an injection $i:X\to\S$. 

We are interested in $[f]\in\rat_d$ so the space that we will consider is not the space of injections but its quotient by the corresponding action: the injection $i$ will be considered modulo post composition by a Moebius transformation.

\begin{definitionint}For a finite set $X$ containing at least three elements, the space $\Mod_X$ is the space of injections of $X$ in $\S$ modulo post composition by Moebius transformations and is called the moduli space of marked spheres.
\end{definitionint}

For example, when $\card X=3$ the space $\Mod_X$ has a unique element because the action of the Moebius transformation on the subsets of three points of $\P^1\C$ is transitive. In particular, $\Mod_X$ is compact in this case.
For $\card X>3$ this is not the case and this space has been well studied by people in algebraic geometry. There is a natural compactification of $\Mod_X$ called the Deligne-Mumford compactification.

\medskip
\noindent{\textbf{Another point of view on holomorphic families.}}

Let's go back to our initial problem but with a different point of view. Consider an holomorphic family 
$$F_t(z):=\frac{a_d(t)z^d+\ldots+a_1(t)z+a_0(t)}{b_d(t)z^d+\ldots+b_1(t)z+b_0(t)},\quad t\in D({0,1}).$$
 For the study of the boundary of $\rat_d$, we will consider that $f_t\in\Rat_d$ if and only if $t\neq0$.

By holomorphic family we mean that the $a_i(t)$ and $b_j(t)$ are Laurent series in $t$ without essential singularities.
The other idea is to consider that instead of having a family of rational map with coefficients in $\C$, we can think of $F_t$ as only one rational map whose coefficients are Laurent series without essential singularities or in another field that contains them. We denote by $\LL$ a field that contains these series and that will be specified when we will really need it. The only thing that the reader should keep in mind is that even if this field is really bigger than what we are used to see in holomorphic dynamics, the only interesting property for us is that it our series. 

Thus what we will be considering is a rational map

$$F_t(z):=\frac{a_dz^d+\ldots+a_0}{b_dz^d+\ldots+b_0}\in \LL(z). $$

The field $\LL$ has the particularity to be equipped by a non-Archimedean norm (to be defined later). In this case, there is an interesting construction existing over this field which is called a Berkovich space. We will see that this space has a natural tree structure and that the map $F_t$ acts naturally on this tree which provides a dynamical system giving information about how the family $[f_t]\in\rat_d$ diverges when $t\neq 0$. 


\medskip
\noindent{\textbf{Outline.}}
This paper will try to be pedagogic. Hence will first make an exposition of the comparison of these two notions emphasizing the ideas and forgetting a little about the Berkovich general formalism that will be introduced in Section \ref{Chap7}. As in order to do dynamics we first need a space, then a map and finally a way to iterate, the structure of this paper will be the sequel: spaces Section \ref{Chap2}, maps Section \ref{Chap3} and dynamics Section \ref{Chap4}. Then in Section \ref{Chap5} we will compare the approaches underlying advantages and disadvantages of these two ones. Finally we will 
have Section \ref{Chap7} as already explained.


\medskip
\noindent{\textbf{Acknowledgments.}} This paper follows from discussions with Jan Kiwi during my PhD. I want to thank also Charles Favre, Mattias Jonsson and Juan Rivera-Letelier for discussion that allowed me to go further in the Berkovich formalism.


\section{Spaces}\label{Chap2}

\subsection{Trees of spheres on an example}

In this section we forgot about rational maps and fix a set $X$ containing at least three elements. As we saw above, we are interested in understanding the diverging sequences in $\Mod_X$.

For simplicity, let us first suppose that $\card X=4$ (we already saw that for ${\card X=3}$ the space is compact). We set $X=\{a,b,c,d\}$.
Consider a sequence $[i_n]$ without converging subsequences in $\Mod_X$. The natural reflex when we are working modulo Moebius transformation, is to post-compose $i_n$ by such a transformation $M_n$ in order to control the position of three points. We call $M_n\circ i_n$ a normalization of the marked sphere.

 For example 
$M_n\circ i_n$ maps $a$ to $0$, $b$ to $1$ and $c$ to $\infty$.
Using the compactness of $\S$, we can suppose (maybe after extracting a subsequence) that the sequence $M_n\circ i_n(d)$ converges. As $[i_n]$ diverges, we deduce that the limit is $0,1$ or $\infty$. Suppose for example that $\lim M_n\circ i_n(d)=\infty= M_n\circ i_n(c)$.

 A naive way to think would be to define the limit of $[i_n]$ to be the application that maps $a$ to $0$, $b$ to $1$ and both $c$ and $d$ to $\infty$ (i.e. the limit of $M_n\circ i_n$). Indeed, if we take another normalization, i.e. another sequence of Moebius transformation $N_n$ that maps now $a$ to $0$, $c$ to $\infty$ and $d$ to $1$, one can see that this process would give a different limit for $[i_n]$.

It is an easy exercise to check that if we choose another sequence of representative such that the image of three elements of $X$ are constant we have a limit with differs only from the two previous one by a Moebius translation. So a first idea is to think that the limit that we define for $[i_n]$ is going to be the collection of these two possible limits.

The second idea is to remark that these two limits are related. Indeed, if for example $M_n\circ i_n(d)=n$ then we can compute that $N_n\circ M_n^{-1}(z)= z/n$ and it follows that $N_n\circ i_n(b)=1/n\to0=N_n\circ i_n(a)$. Hence this collection of limits comes with some rigidity. This is why we introduce trees: in order to give a simple visualization of this rigidity. In this example we identify the limit of $[i_n]$ to a tree of sphere (that we will define later) as on  Figure \ref{dmpp}. 

 \begin{figure}
  \centerline{\includegraphics[width=14cm]{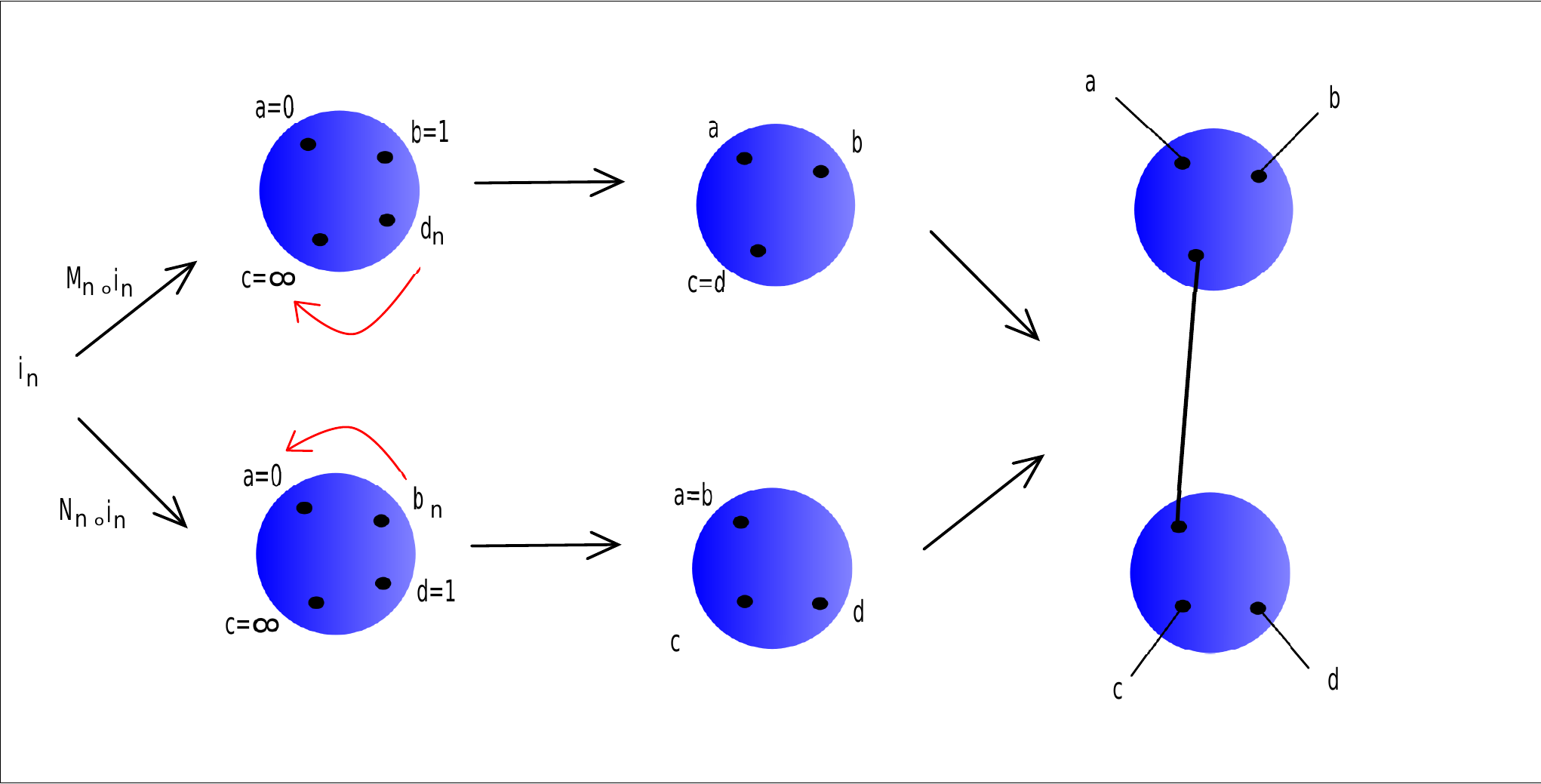}} 
   \caption{Illustration of the convergence to a tree of spheres.}
\label{dmpp} \end{figure}

\begin{remark} \label{zoomstof}
When we consider the normalization $M_n\circ i_n$, we see that the images of $c$ and $d$ are collapsing so intuitively, considering the normalization $N_n\circ i_n$ corresponds to zooming or rescaling at the level of these two points. 
\end{remark}
 
 Before developing the definition of trees of spheres in the next subsection, let us first describe this example.
The tree of sphere on Figure \ref{dmpp} has 5 edges, and 6 vertices. There are two categories of vertices: elements of $X$ at the ends (or leaves) of the tree and spheres at the interior and on which are attached edges. At each internal vertex $\S_v$ the set of branches induces a partition of $X$. Hence this naturally provides an application $X\to\S_v$ that maps each element of $x$ to the attaching point of the branch to which it belongs. The two spheres provide two applications that correspond exactly to the ones associated to $M_n$ and $N_n$. Note that from the definition, both of the spheres have at least three different marked points. Of course, in order to have a well define and unique limit we have to consider these trees modulo a natural action of isomorphism.

\subsection{Limits and trees of spheres}

As we saw above, the structure of the trees we are going to define is essentially combinatorial. It is important as it encodes a rigidity that is not obvious to describe with a simple vocabulary. That's why we define everything in order to make easy to talk about it.

A graph is the disjoint union of a finite set $V$ called set of vertices and an other finite set $E$ consisting of elements of the form $\{v,v'\}$ with different $v,v'\in V$.
We say that $\{ v,v'\}$ is an edge between $v$ and $v'$. For all $v\in V$ we define $E_v$ the set of edges containing $v$.
We call valence of $v$ the cardinal of $E_v$.

\begin{definition} A (stable) tree is a connected graph without cycle and whose internal vertices have at least valence $3$. 
\end{definition} 

For every vertex $v$ and every edge $e$ adjacent to $v$, we denote by $B_v(e)$ the branch on $v$ that contains $e$, ie the connected component of $T\setminus\{v\}$ that contains $e$.


\begin{definition}
A tree of sphere ${\cal T^X}$ (marked by $X$) is  the data of: 
\begin{itemize}
\item a combinatorial tree $T^X$ whose leaves are the elements of $X$ and 
\item for every internal vertex $v$ of $T^X$, \begin{itemize}
    \item a Riemann sphere $\S_v$ and
    \item a one-to-one map $i_v:E_v\to \S_v$. 
    \end{itemize}
\end{itemize}
\end{definition}

We denote by $\overline{\rm Mod}_X$ the space of trees of spheres marked by $X$.
For $e\in E_v$, we say that $i_v(e)$ is the attaching point of $e$ on $v$. We extend $i_v$ to $T\setminus \{v\}$ by setting $i_v(v'):=i_v(e)$ if $v'\in B_v(e)$. We denote by $a_v$ the application $i_v|_X$.
When the tree has a unique internal vertex $v$, we identify it to the marked sphere $a_v$.

\begin{definition}
A sequence of marked spheres $a_n:X\to{\mathbb S}_n$
  converges to a tree of sphere ${\cal T}^X$
  if for all internal vertex $v$
 of $\T^X$ , there exists an isomorphism $M_{n,v}:{\mathbb S}_n\to{\S}_v$
    such that $M_{n,v} \circ a_n$
     converges to $a_v$. 
     \end{definition}
(We prefer to use the notation $\S_n$ instead of $\S$ or $\P^1\C$ because the $\S_n$ should be thought distinct.)

We keep in this paper this notion of convergence witch is not Hausdorff. In fact this notion is Hausdorff after passing to a quotient by the appropriated notion of isomorphism of trees of spheres. We will not develop this notion of isomorphism in this paper but the interested reader can take a look to \cite{A3}.

\subsection{Berkovich space}

Recall that $\LL$ is a special field that contains Laurent series without essential singularities. We are going to explain how is constructed a Berkovich space over $\LL$. For simplicity we assume in this section that $\LL$ is the set of these series.  Let us define a norm on this space. Let
$$c:=\sum_{k\geq k_0}c_{j} t^{k}\in \LL\quad \text{with}\quad k_0\in\Z\text{ and }c_{k_0}\neq0.$$

We define $$|c|_\LL:=exp({-k_0}).$$
Setting $|0|_\LL=0$, we assure that $|.|_\LL$ is a norm on $\LL$. Moreover, it satisfies an inequality stronger then the usual triangular inequality, called the ultrametric or non-Archimedean  inequality and that defines $\LL$ to be a non-Archimedean field:
 $$\forall a,b\in\LL, |a+b|_\LL\leq \rm{max}(|a|_\LL,|b|_\LL).$$ 

This inequality, is very strange for people who are not used to it. For example I have here to say that this paper design by "closed" (resp. "open") ball the ball define by a large (resp. strict) inequality because both of them are open and closed for this topology.
Also, if we denote by $B(z,r)$ the closed ball of center $z\in\LL$ and radius $r\geq0$, we have the surprising following lemma.

\begin{lemma}\label{ball} Given $z,z'\in\LL$ and $r\geq0$, we have
$$z\in B(z',r)\implies B(z,r)=B(z',r).$$
\end{lemma}

\begin{proof}
Suppose that $|z-z'|\leq r$. Then
$$x\in B(z,r)\implies |x-z'|_\LL\leq \max (|x-z|_\LL,|z-z'|_\LL)\leq r\implies x\in B(z',r).$$
\end{proof}

This means that every element in a closed ball is the center of this ball and, as a consequence, that the intersection of two balls is empty or one ball is included in the other. It will play a major role in the understanding of tree structure in the Berkovich space.

To define precisely what is this Berkovich space, we would have to talk about valuations, semi norms but one would want first to forget about the formalism and concentrate on the keys for understanding this space. Let's denote by $\P^1_\LL\approx\LL\cup\{\infty\}$ the projective space over $\LL$.

\begin{definition}
  $ \P^1 _{Berk}$ is a compactification of $\P^1_\LL$ called the Berkovich projective line.
 \end{definition}
 
Here is the important remark due to Berkovich:

\begin{proposition}\label{prethm}
Almost every element of $ \P^1 _{Berk}$ is identified to a closed ball 
$$B(z_0,r)=\{z\in \P^1_\LL ~|~ |z-z_0|\leq r \}.$$
\end{proposition}

From this we can deduce that $\P^1 _{Berk}$ has a real tree structure.
A way to represent the space $\P^1 _{Berk}$ is to consider it as the quotient of the product $" {\LL} \times (\R^+\cup\{\infty\} )"$ (Center$\times$ Radius) by the relation provided in Lemma \ref{ball}. The real structure of this space comes from the vertical line that all behave well to the quotient.

For example take $a\neq b$ in $\LL$. Consider the balls $B(a,r)$ and $B(b,r)$ for the radii $r\in[0,\infty]$. Define $r_0:=|b-a|_\LL$. For $r<r_0$ the balls $B(a,r)$ and $B(b,r)$ are disjoint but for $r\geq r_0$ they are equal according to Lemma \ref{ball} (cf Figure \ref{Berkline}).

Hence we have to think about $\P^1 _{Berk}$ as real lines attached to
the elements of $\LL\cup\{\infty\}$, identified to the balls of radius $0$ or infinity, parametrized by the radius of the balls and with intersecting points at each values of $|a-b|_\LL$ for some $a$ and $b$ in $\LL$ (cf Figure \ref{P1brek}).

\begin{remark} Proposition \ref{prethm} gives a good idea of what is the all space. For the purpose of proving that trees of spheres are included in the Berkovich space, it will be sufficient to consider the elements mentioned in this statement. We will give in Section \ref{Chap7} a complete version of this theorem.The points that corresponds to balls of radius 0 are identified to elements of $\P^1\LL$  and are called type I. The non branching points are said to be type II and the branching ones are type III. The points that we omitting are some ends of the Berkovich tree. These points are called the type IV points. We will totally abuse and design in this paper by "end" only the type I points.
\end{remark}

\begin{figure}[h]
\begin{minipage}[c]{.45\linewidth}
\begin{center}

\includegraphics[width=6cm]{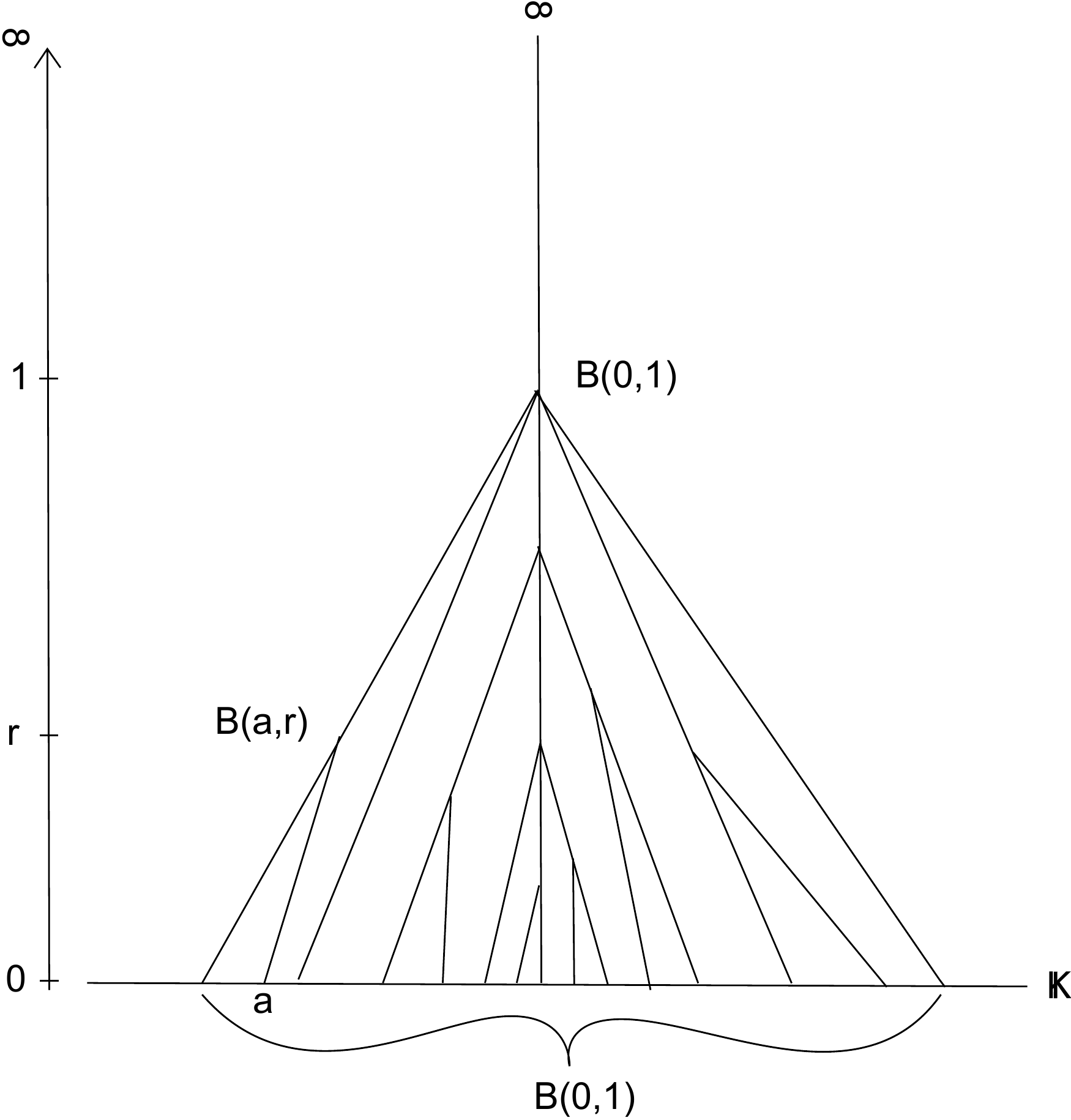}  \caption{}\label{Berkline}

\end{center}
\end{minipage}
\hfill
\begin{minipage}[c]{.45\linewidth}
\begin{center}

\includegraphics[width=6.5cm]{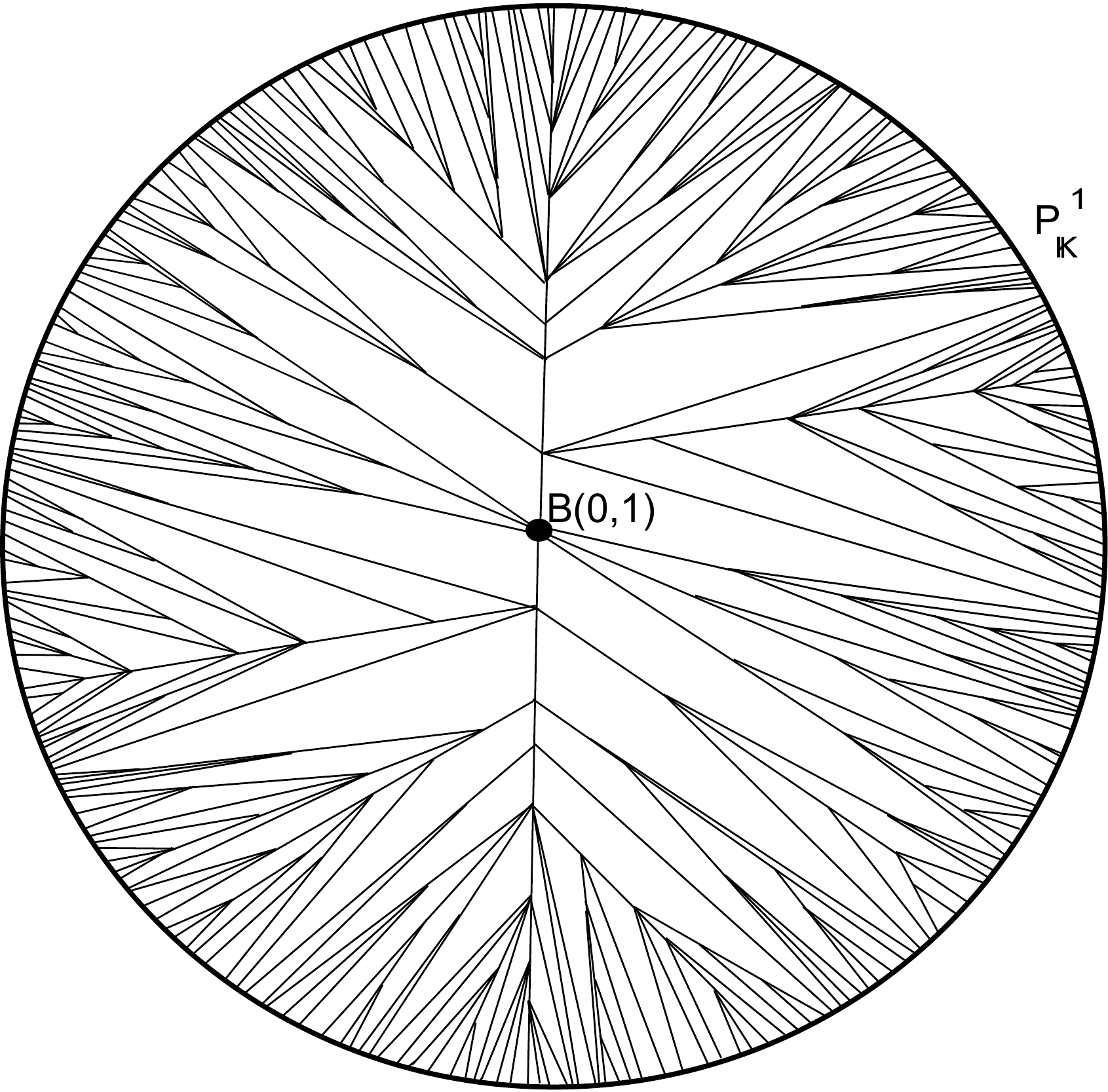}  \caption{ }\label{P1brek}

\end{center}
\end{minipage}
\end{figure}

\subsection{Limits and Berkovich space}

Now, let us look at the branching points in $ \P^1 _{Berk}$. We saw that they occur at some points $B(a,|a-b|_\LL)$ for some $a$ and $b$ in $\LL\cup\{\infty\}$. Up to a change of coordinate in $\P^1_\LL$, we can consider that this ball is the (closed) ball $${\zeta}:=B(0,1)$$ called the Gauss point.

This point is very particular from an algebraic point of view as we will see the formalism in detail. The elements of $\zeta$ have the following form:
$$c=\sum_{k\geq0}c_{k} t^{k} \text{ with all } c_k\in \C.$$
We define a map $\rho:\P^1 _{Berk}\to\P^1\C$ by $\rho(c)= c_0$ for $c\in\zeta$ and $\rho(c)=\infty$ for $c\notin \zeta$. Note that when it is well defined, applying $\rho$ corresponds to taking the limit.

\begin{lemma} Two elements of $\P^1\LL$ are in the same branch of $\zeta$ if and only if they have the same image by $\rho$. This gives a canonical identification between the branches at $\zeta$ and $\P^1\C$.
\end{lemma}

\begin{proof} For $d\in\P^1\LL\setminus\zeta$, we have $|d-0|_\LL>1$ so 
$$\zeta\subset B(0,|d|_\LL)=B(0,|d|_\LL)$$
 and all the elements $ B(d,0)\approx d$ with $|d|_\LL>1$ are in a same branch containing $\infty$.

For $a,b\in\zeta$, we saw before that $B(a,r)=B(b,r)\subset \zeta$ is a branching point separating $a\approx B(a,0)$ and $b\approx B(b,0)$. Thus $a$ and $b$ are in the same branch if and only if $|b-a|_\LL<1$. The result follows by remarking that $|b-a|_\LL<1$ if and only if the corresponding series have the same constant term.
\end{proof}

We denote the set of branches at $\zeta$ by $T_{\zeta} \P^1 _{Berk}$. It is usually called the tangent space at $\zeta$. More generally, for any branching point $a\in\P^1_{Berk}$, we denote $T_{\zeta} \P^1 _{Berk}$ its tangent space, ie the set of branches at $a$.

\begin{remark}\label{remopenball}
Note that the intersection every branch at a branching point $B(a,r)$ (for $a\in\P^1_\LL$ and $r>0$) with $\P^1\LL$ is empty or an open ball. Indeed, after a change of coordinate we can suppose that this branch $B$ does not contains $\infty\in\P^1\LL$. Take $a'\in B\cap\LL$. Then $B(a,r)\cap\LL$ is the increasing union of the closed balls $B(a',r')$ for $0\leq r'<r$.
\end{remark}

\subsection{Comparison}\label{compspace}

In order to be able to compare these two worlds, we have to first consider families of marked spheres instead of sequences. We suppose that for all $t\in D(0,1)\setminus \{0\}$, the maps $i_t:X\to\S_t$ are injections from the finite set $X$ to some spheres that we identify to the Riemann sphere $\C\cup\{\infty\}$. For every $a\in X$ we denote by $a_t$ the element $i_t(a)$. We suppose that the $a_t$ depend holomorphicaly  on $t$.

Then naturally, every $a_t$ for $a\in X$ is identified to an element of $\LL$.
Now take two other elements $b$ and $c$ in $X$ and identify them as for $a$ to elements $b_t$ and $c_t$ in $\LL$. Note that $a_t, b_t$ and $c_t$ are also identified to ends of the tree $\P^1_{Berk}$, thus there is a unique branching point $v\in \P^1_{Berk}$ separating them. Denote by $\S_v$ the tangent space $T_v\P^1_{Berk}$. 
Applying a Moebius transformation $M_{v,t}$ with coefficients in $\LL$ we can suppose that $a_t=1$, $b_t=1$ and $c_t=\infty$. Then $v=\zeta$ so we can deduce that $\rho(a_t)=\lim_{t\to0}a_t$. 
It follows that two elements of $M_{v,t}\circ i_t(X)$ have the same limit on $\S_v$ if and only if they are in the same branch of $\S_v$. 

Every sphere $\S_v$ in $\T^X$ separates at least three elements of $X$, so for each sphere we can associate a vertex $v$ in $\P^1_{Berk}$ and its tangent space to $\S_v$ and the notions of convergence corresponds: for every vertex in $\P^1_{Berk}$ corresponding to a vertex in $\T^X$, there is a Moebius transformation $M_t$ such that $a$ and $b\in X$ are in the same branch if and only if $M_t\circ i_t(a_t)$ and $M_t\circ i_t(b_t)$ have same limit.

\begin{remark}\label{actionMoeb} We can define an action of the space of Moebius transformations of $\P^1\LL$ on the branching points of $\P^1_{Berk}$. Indeed, given a branching point $v$, it separates at least three ends of the tree that we can identify to points of $\LL$. Then there exists a unique Moebius transformation $M_t$ that send these three points to $0,1$ and $\infty$. We define the image of $\zeta$ by $M_t$ to be $v$ and this is sufficient to define a transitive action on the all space of branching points of $\P^1_{Berk}$.

Because of this and as we saw in Remark \ref{zoomstof} that a sequence (or a family in this setting) of Moebius transformations can be interpret as a rescaling, we can think of the set of branching points of $\P^1_{Berk}$ as a space of all the possible zooms or rescalings. In fact we could go much further and explicit a relation between these Moebius transformations and blows-up.
\end{remark}


\section{Maps}\label{Chap3}

\subsection{Trees of spheres covers}

We want to introduce maps between trees of spheres. For this it is convenient to first look at this notion in the case of marked spheres.
Recall that the motivation to look at marked spheres was to mark critical points and periodic points. Hence, if we wanna consider sequences of rational maps and consequently sequences of marked spheres, it is clear that all of the maps in the same sequence should have the same number of critical points and the same number of other marked points, ie the same combinatorial datas.

We fix two finite sets $Y$ and $Z$ with at least three elements.
We define a portrait ${\bf F}$ of degree $d\geq 2$ to be a couple $(F,\deg)$ where 
\begin{itemize}
\item $F:Y\to Z$ is a map between two finite sets $Y$ and $Z$ and
\item $\deg:Y\to \N-\{0\}$ is a function that satisfies
\[\sum_{a\in Y}\bigl(\deg(a) -1\bigr) = 2d-2\quad\text{and}\quad \sum_{a\in F^{-1}(b)} \deg(a) = d\quad\text{ for all } b\in \Z.\] 
\end{itemize}
Typically, $Z\subset \S$ is a finite set, $F:Y\to Z$ is the restriction of a rational map $F:\S\to \S$ to $Y:=F^{-1}(Z)$ and $\deg(a)$ is the local degree of $F$ at $a$. In this case, the Riemann-Hurwitz formula and the conditions on the function $\deg$ implies that $Z$  contains the set $V_F$ of the critical values of $F$ in order to let $F:\S-Y\to \S-Z$ be a cover. 

\begin{definition}
A rational map marked by ${\bf F}$ is a triple $(f,y,z)$ such that $f\in \Rat_d$, $y:Y\to \S$ and $z:Z\to \S$ are marked spheres, and the following diagram commutes: 

\centerline{
$\xymatrix{
 Y \ar[r]^{y}    \ar[d] _{{ F}} &\S  \ar[d]^{f} \\
      Z \ar[r]_{z}  &\S
  }$} with $\deg_{y(a)}f = \deg(a)$ for all $a\in Y$. 
\end{definition}

Now we want to consider a sequence $(f_n,i_n,j_n)_n$ of rational maps marked by the same portrait ${\bf F}$ and understand how the trees of spheres are going to help us to define an interesting limit to it. The main idea comes from the following lemma:

\begin{lemma}\cite{A2}\label{postcomp}
Let $(f_n: \S \to  \S )_n$ be a sequence of rational maps of same degree. Then, there exists a subsequence $(f_{n_k})_{n_k}$ and a sequence of Moebius transformations $(M_{n_k})_{n_k}$ such that $ (M_{n_k}\circ f_{n_k})_{n_k}$ converges to a non constant rational map $f$  uniformly outside a finite number of points.
\end{lemma}

This mean that a sequence of rational map is never really diverging to a constant but its natural image can disappears to another scale. After passing to a subsequence $i_n$ converges to a tree of sphere $\T^Y$. In our case, each sphere $\S_v$ of $\T^Y$ corresponds to a scale $M_{n,v}$ and another map corresponding to this scale which is $f_n\circ M_{n,v}^{-1}$. This lemma insure that by looking at the right scale for the image we get a non constant limit. 

Now, the magic of the right definitions works and it follows that if $j_n$ converges to a tree of spheres $\T^Z$, then the scales represented corresponding to the spheres of $\T^Z$ are exactly the ones given by Lemma \ref{postcomp}.


\begin{theorem}\label{toscover} Let $(f_n,i_n,j_n)_n$ be a sequence of rational maps marked by the same portrait ${\bf F}$.
If 
$([i_n],[j_n])\to([{\cal T}^Y],[{\cal T}^Z])\in \overline{\rm Mod}_Y\times \overline{\rm Mod}_Z,$
then $(f_n,i_n,j_n)_n$ converges to a cover between trees of spheres $${\cal F}:{\cal T}^Y\to {\cal T}^Z.$$

\end{theorem}

Where covers between trees of spheres and convergence are defined below.

\begin{definition}[Cover between trees of spheres] 
 ${\cal F}:{\cal T}^Y\to {\cal T}^Z$ is : 
\begin{itemize}
\item a trees map $F:T^Y\to T^Z$ 
\item for $w:=F(v)$, 
a ramified cover $f_v: \S_v\to \S_w$ s.t.
\vskip0.2cm
\begin{enumerate}
\item the restriction $f_v: \S_v-a_v(Y)\to \S_w-a_w(Z)$ is a cover;
\vskip0.2cm
\item adjacent attaching points maps to the corresponding adjacent attaching points
\vskip0.2cm
\item same local degree on both sides of edges.  
\end{enumerate}
\end{itemize}
\end{definition}

\begin{definition}[Convergence]
Let ${ \F}:{\T}^Y\to { \T}^Z$ be a cover between trees of spheres of portrait ${\bf F}$. A sequence ${ \F}_n:=(f_n,a_n^Y,a_n^Z)$ of marked spheres covers converges to ${ \F}$ if their portrait is ${\bf F}$ and if for all pair of internal vertices $v$ and $w:=F(v)$, there exists sequences of isomorphisms $M_{n,v}^Y:\S_n^Y\to \S_v$ and $M_{n,w}^Z:\S_n^Z\to \S_w$ such that 
\begin{itemize}
\item $M_{n,v}^Y\circ a_n^Y:Y\to \S_v$ converges to $a_v^Y:Y\to \S_v$, 
\item $M_{n,w}^Z\circ a_n^Z:Z\to \S_w$ converges to $a_w^Z:Z\to \S_w$ and 
\item $M_{n,w}^Z\circ f_n\circ (M_{n,v}^Y )^{-1}:\S_v\to \S_w$ converges locally uniformly outside $Y_v$ to ${f_v:\S_v\to \S_w}$. 
\end{itemize}
\end{definition}


\begin{remark}These covers have natural properties such as
\begin{itemize}
\item combinatorial surjectivity (every sphere of $\T^Z$ has a preimage),
\item global degree (the global number of preimages of every point on a sphere of $\T^Z$ is constant) and
\item allowing Riemann-Hurwitz type formula for a well chosen equivalent of the Euler characteristic.
\end{itemize}
\end{remark}


\subsection{Maps on Berkovich space}

For every polynomial $P_t:=X^d+a_1.X^{d-1}+\ldots+a_{d}\in \LL[X]$  we have the explicit formula:
$$P(B(0,r))=B(a_d(t),\underset{i\in[1,d]}{max}(|a_i(t)|_\LL . r^{d-i})).$$
Hence $P_t$ defines a dynamic on $  \P^1 _{Berk}$. 

However the case of rational maps is a little more difficult but we can also associate a unique ball as an image to another one and thus define a dynamic on $\P^1 _{Berk}$. 
For this we have to remark that the map $F_t$ is well defined on $\P^1\LL$, and that for each branching point$a\in \P^1 _{Berk}$ all the open balls corresponding to the branches of $a$ (see Remark \ref{remopenball}) but a finite number of them map to open balls corresponding to the branches of a unique branching point of $\P^1 _{Berk}$. This is unfortunately not easy to prove, we would have to prove the following that we admit in this paper (cf \cite{Juan}).

\begin{lemma}
If an open ball maps to another open ball, then the corresponding branch maps to the other one.
\end{lemma}

Thus, for a branching point $v$, this define almost everywhere a map 
$$D_{v}F_t:T_{v} \P^1 _{Berk}\to T_{F_t(v)} \P^1 _{Berk}.$$

\begin{example} cf Remark \ref{actionMoeb}.
\end{example}

As tangent spaces are related to limits, it is time to look at the meaning in terms of the map $\rho$. Recall that we defined a map $\rho:\LL\to\C$ which is the equivalent of taking the limit $t\to 0$. This map extends to $\rho:\LL[X]\to\C[X]$ by applying $\rho$ to the coefficients. In order to extend it to a map $$\rho:\LL(X)\to\C(X),$$ we have to be more careful and for each $F_t=P_t/Q_t$ with $P_t,Q_t\in \LL[X]$ we first have to simplify $P_t$ and $Q_t$ by the maximal power of $t$ that can be factored for both of them and, then, we can apply $\rho$ to the coefficients. Of course this is the equivalent of taking the limit when $t\to 0$ when it is meaningful.

Now we are ready to state the most important statement about functions acting on the Berkovich space.

\begin{theorem}\label{thmzeta}
 If $F_t(\zeta)=\zeta$, the tangent map $D_{\zeta}F_t:T_{\zeta} \P^1 _{Berk}\to T_{\zeta} \P^1 _{Berk}$
is a well defined non constant rational map and satisfies
 $$ D_{\zeta}F_t=\rho(F_t).$$ 
\end{theorem}

Hence we have the following corollary.

\begin{corollary}
If $deg(\rho(P_t)/\rho(Q_t))>0$ and $(F_t)_{t\in D(0,1)-\{0\}} $ is a holomorphic family in $\Rat_d$ with $F_t:\P^1_\LL\to\P^1_\LL$, then
$$F_t=\frac{P_t}{Q_t}\;
{\longrightarrow}\;\frac{\rho(P_t)}{\rho(Q_t)}=\rho(F_t).$$
\end{corollary}

In fact the convergence is local, uniform outside a finite set.


\subsection{Comparison}

Consider an holomorphic family 
$$f_t(z):=\frac{a_d(t)z^d+\ldots+a_1(t)z+a_0(t)}{b_d(t)z^d+\ldots+b_1(t)z+b_0(t)},\quad t\in D({0,1}),$$
such that $f_t\in\Rat_d$ if and only if $t\neq0$.

Fix a portrait ${\bf F}$ of degree $d$ and suppose that for $t\neq 0$, we can associate to $F_t$ two marqued spheres $i_t$ and $j_t$ as in the section \ref{compspace} and such that $(\F_t,i_t,j_t)$ is a rational map marked by ${\bf F}$. Then $i_t$ and $j_t$ converge respectively to some trees of spheres $\T^Y$ and $\T^Z$ that can be identified in $\P^1_{Berk}$ and, according to Theorem \ref{toscover}, there is a trees of spheres cover $\F:\T^Y\to\T^Z$.

Given a vertex $v\in T^Y$, there exist a moebius transformation $M_{v,t}$ that maps $\zeta$ to $v$ and $M_{F(v),t}$ the maps $\zeta$ to $F(v)$. Hence the map $$F _{v,t}:=M_{F(v),t}\circ F_t\circ M_{v,t}$$ fixes $\zeta$ and Theorem \ref{thmzeta} together with its corollary assures that $f_v$ and $D_\zeta F _{v,t}:T_{\zeta} \P^1 _{Berk}\approx\S\to T_{\zeta} \P^1 _{Berk}\approx\S$ are conjugated.


\subsection{Remarks about quotient}

Now we are ready go back to the original spaces. The sequences or families that we were studying are in $\rat_d$.
On the one had, for the trees of spheres point of view, we see that considering a sequence of dynamically marked rational map $(M_n\circ f_n\circ M_n^{-1},M_n\circ i_n,M_n\circ j_n)_n$ instead of $(f_n,i_n,j_n)_n$ does not really affect the notion of convergence as in that definition allows to pre-compose and post-compose $f_n$ when we want to prove that it converges in certain charts to the corresponding applications between two spheres.

On the other hand, for the Berkovich point of view, considering ${M_t\circ F_t\circ M_t^{-1}}$ instead of $F_t$ corresponds to making a change of coordinate in $\P^1\LL$ and by consequence changing the corresponding point $\zeta$. The new $\zeta$ corresponds to the image by $M_t$ of the old one in the corresponding coordinates.


\section{Dynamics and rescaling limits}\label{Chap4}

\subsection{Dynamics on trees}

The dynamics in the branching points of the Berkovich space is already well defined. However, it needs more work for the case of trees of spheres. Indeed, when we mark a rational map, we distinguish the domain of definition and the image of the map. For example, already at the level of a portrait, there is no meaning of cycle. For this we would need to have at least an identification of a subset of $Z$ into a subset of $Y$.
Hence we consider the case where there exists a set $X= Y\cap Z$ with at least three elements.
To be consistent, we say that a triple $(f,i,j)$ is a rational map dynamically marked by $({\bf F},X)$ if ${\bf F}$ is a portrait, $i|_X=j|_X$ and the following diagram commutes

{ $$\xymatrix{
    X\ar[r]\ar[rd]&Y \ar[r]^i \ar[d]_{ F}  & \S \ar[d]^f \\
    &Z \ar[r]_j & \S.
  }$$}

Before defining dynamics on trees of spheres, let us already go back to the Berkovich space point of view. Suppose that the $(f_t,i_t,j_t)$ are as in the previous section and that in addition they are rational map dynamically marked by $({\bf F},X)$. Then as before we have $\T^Y$ and $\T^Z$ in the Berkovich space. But now we have an additional information : the ends corresponding to the elements marked by $X$ are common to the trees $\T^Y$ and $\T^Z$. Thus we have to identify the spheres of the trees of spheres $\T^Y$ and $\T^Z$ that corresponds to branching points separating three elements of $X$.

\begin{definition}
A tree of spheres ${\cal T}^X$ is compatible with a tree of spheres ${\cal T}^Y$ if 
\begin{itemize}
\item $X\subseteq Y$, $IV^X\subseteq IV^Y$, 
\item for all internal vertex $v$ of $T^X$, we have 
\begin{itemize}
\item $\St^X_v=\St^Y_v$ and
\item$a_v^X=a_v^Y|_X $. 
\end{itemize}
\end{itemize}
\end{definition}

We write $\T^X\lhd \T^Y$ in this case. 

\begin{definition}
A dynamical system of trees of spheres is a pair $(\F,\T^X)$ such that
\begin{itemize}
\item ${\cal F}:{\cal T}^Y\to {\cal T}^Z$ is a cover between trees of spheres ;
\item $\T^X\lhd \T^Y$ and $\T^X\lhd\T^Z$. 
\end{itemize}
\end{definition}

We can define dynamics and iterate the spheres of $\T^Y$ as soon as their iterates lie in $\T^X$. If this is not the case, then we have to stop iterating, which is one of the inconvenient of this formalism.
We also have to adapt a little the notion of convergence.

\begin{definition}\label{defcvdyn}
Let $({\F}:{ \T}^Y\to { \T}^Z,{ \T}^X)$ be a dynamical system of trees of spheres with portrait ${\bf F}$. A sequence $(f_n,y_n,z_n)_n$ of dynamical systems between spheres marked by $({\bf F},X)$
converges to $(\F,\T^X)$ if  
$$\displaystyle (f_n,y_n,z_n)\underset{M_n^Y,M_n^Z}\longrightarrow{ \F}\quad\text{with}\quad M_{n,v}^Y=M_{n,v}^Z$$ for all vertex $v\in IV^X$.
\end{definition}

\subsection{Rescaling limits}

Let us now see what dynamics on these trees can tell about the divergence of the sequence of rational maps $(f_n,i_n,j_n)$ dynamically marked by a same portrait. Suppose that we have a sphere $\S_v$ of $\T^X$ of period $p$. We deduce from the definitions that the map
$$ (M^Z_{n,F^p(v)}\circ f_n\circ\ldots M^Z_{n,F^2(v)}\circ f_n\circ (M^Y_{n,F(v)})^{-1})\circ (M^Z_{n,F(v)}\circ f_n\circ (M^Y_{n,v})^{-1})$$
which is the same as the map $ M^Y_{n,v}\circ f^p_n\circ (M^Y_{n,v})^{-1}$
converges locally uniformly outside a finite number of point to a non constant map.

For a period $p$ branching point in $\P^1_{Berk}$ for some $F_t$, it is even easier: if $M_{t,v}$ maps $v$ to $\zeta$ then $M_{t,v}\circ F_t\circ M^{-1}_{t,v}$ fixes $\zeta$ so converges also locally uniformly outside a finite number of point to a non constant map.

These phenomena are meaningful from the dynamical point of view and such limits are called rescaling limits. As we can see, these limits come as cycles so there is a natural notion of dependance that follows from these (see \cite{K2} for the details). It is not the purpose of this paper to study rescaling limits. The interested reader can refer to \cite{K2} or to \cite{A1} and \cite{A2} for the results known today.


\section{Comparison of these approaches}\label{Chap5}

\subsection{Overview of the bridge}
Before comparing these two approaches, let's first resume the links that we proved between these two approaches.

First, we suppose that for all $t\in D(0,1)\setminus \{0\}$, the maps $i_t:Y\to\S_t$ and $j_t:Z\to\S_t$ are injections with finite sets $Y$ and $Z$ whose intersection $X$ contains at least 3 elements. For every element $a$ in these sets we denote by $a_t$ its image on $\S_t$ and we suppose that the $a_t$ depend holomorphicaly  on $t$ and that $i_t$ and $j_t$ converge to the respective trees of spheres $\T^Y$ and $\T^Z$. 

The elements in $i_t(Y)$ and $j_t(Z)$ are points in $\P^1\LL$ identifiable to ends in $\P^1_{Berk}$. 
Every internal vertex $v$ in $\T^Y$ (resp. $\T^Z$) separate three points in $Y$ (resp. $Z$) and corresponds to a unique branching points $B_v$ in $\P^1_{Berk}$. There exists Moebius transformation $M_{t,v}$ (not uniquely defined) that maps $\zeta$ to $B_v$. The spheres $\S_v$ and $T_{B_v}\P^1_{Berk}$ are identified.
The composition $M_{t,v}\circ i_t$ (resp. $M_{t,v}\circ j_t$) converges to the marking $a_v$ of $\S_v$ (modulo maybe a post composition by a constant Moebius transformation).

Consider an holomorphic family $F_t$ of rational map of degree $d$ when $t\neq0$, a portrait ${\bf F}$ and suppose that $(\F_t,i_t,j_t)$ is a rational map marked by ${\bf F}$. then there is a trees of spheres cover $\F:\T^Y\to\T^Z$. If $F(v)=w$ then $F_t(B_v)=B_w$ and $f_v=\rho( M_{w,t}^{-1}\circ F_t\circ M_{v,t})$ (up to a pre and post composition by some constant Moebius transformations).


\subsection{Advantages and disadvantages}

There are three advantages to use the trees of spheres point of view.
First, the vocabulary is quite simple and really matches the concern of people in holomorphic dynamics. 
Second, the proofs for the rescaling theorems are more combinatorics and topology: the analytic properties are hidden inside lemmas (see Branches and Annuli lemmas in \cite{A2}). 
Third, given a portrait ${\bf F}$, the space of rational maps dynamically marked by ${\bf F}$ comes with a topology described in \cite{A3} and the convergence arise in a nice analytic space when in the Berkovich point of view we provide a dynamics on a space only when we fix a diverging analytic family diverging, i.e. there is no topology and the way to diverge is fixed. Note that Jan Kiwi pointed out that the analytic condition is not a very big constraint as he proved that if a sequence for rational maps has a finite number of rescaling limits, then there exists a corresponding analytic family of rational maps that have the same behavior in the sense defined in Proposition 6.1 of \cite{K2}.

The price to pay in order to get a nice topology is to have to consider finite objects and force the dynamics to appear when in the Berkovich space the dynamics is directly well defined. The other problem of finiteness of these objects is that we have to make a choice of the points we are marking whereas in the Berkovich space in a certain sense all the possible analytical markings are already there. In a some sense, we can say that Berkovich spaces are useful to find an information whereas the trees of spheres are a more natural setting in order to write a convergence when we already have the informations we are interested in.

The other problem to force finiteness of the combinatorics is that the trees of spheres can just have a finite number of rescaling cycles. In \cite{K2}, J. Kiwi remarked that we don't know so far any example of sequences of rational maps that have infinitely many rescaling limits that are dynamically interesting (i.e. of degree at least two and not monomial). However, recent results from Cui G. and Peng W. in \cite{CP} involving the general idea of Shishikura trees (see \cite{Sh1} for a special case) let think that we could produce such examples. 

\section{The formalism of Berkovich spaces}\label{Chap7}


\medskip
\noindent{\textbf{The non-Archimedean field.}}

The field $\LL$ we consider is the completion of $\C\langle\langle t \rangle\rangle$ 
(the algebraic closure of the field of formal Puiseux series). Its elements looks like
$$a=\sum_{q_j\geq q_0}c_{j} t^{q_j}\in\L-\{0\}\quad\quad (c_{j}\in \C,\;c_0\neq0, \;\Q\ni q_j\nearrow\infty).$$ 

The non-Archimedean norm on this field corresponds to the vanishing order at $0$, i.e. $$|a|_\LL:=\rm{exp}(-q_0).$$
We denoted by $ \P^1 _\LL$ the projective space over $\LL$.

\medskip
\noindent{\textbf{The Berkovich projective line.}}

A multiplicative semi-norm on $\LL[X]$, is a function defined on ${|.|_x:\LL[X]\to\R^+}$ that satisfies:
\begin{enumerate}
\item$|0|_x=0$ ;
\item$|1|_x=1$ ;
\item$|fg|_x=|f|_x|g|_x$ and
\item$|f+g|_x\leq |f|_x+|g|_x$.
\end{enumerate}

\begin{example}\label{sminorm}
 Here are two fundamental example of such semi-norms::
\begin{itemize}
\item for every $a\in\LL$, we set $$P\to|P|_a:=|P(a)|_\LL,$$
\item for every $z_0\in \LL$ and every $r\in\R^{\star+}$ we set 
$$B(z_0,r):=\{z\in\LL~|~|z-z_0|_\LL\leq r\}$$ and
$$P\to|P|_{B(z_0,r)}=\underset{z\in B(z_0,r)}{max}|P(z)|_\LL.$$
\end{itemize}
\end{example}

We define the analytical Berkovich space $\mathbb A^{1,{\rm an}}_\LL$ to be the set of multiplicative semi-norms on $\LL[X]$ whose restriction to $\LL$ is $|.|_\LL$ and it is equipped with the weak convergence topology. This space is called the Berkovich affine line. The following theorem due to Berkovich is the right formulation of Proposition \ref{prethm} (see \cite{BR}):

\begin{theorem}\label{thmberk0}
Every element $|.|_B$ of  $\mathbb A^{1,{\rm an}}_\LL$ is realizable as:
$$|P|_B=\lim_{i\to \infty}    \left( \underset{z\in B(z_i,r_i)}{max}|P(z)|_\LL\right).$$
for a decreasing sequence of ("closed") balls:
$$B(z_0,r_0)\supseteq B(z_1,r_1)\supseteq B(z_2,r_2)\supseteq\ldots\quad.$$
\end{theorem}

In particular, if the intersection of the $B(z_i,r_i)$ is not empty we are in one of the examples cited in Example \ref{sminorm}. Otherwise, the point is said to be type IV point and it corresponds to a limit of the other points.


The projective line $\P^1 _{Berk}$ is obtained from $\mathbb A^{1,{\rm an}}_\LL$ by adding the function mapping all non identically zero polynomial  to $\infty$ and $z\in\LL$ to $|z|_\LL.$ It follows that this space is compact.

\medskip
\noindent{\textbf{Reduction.}} 

On $\mathbb A^{1,{\rm an}}_\LL$, there is a point that plays a an important role from an algebraic point of view: $\cal O_\LL:=\{  z\in\LL~|~|z|_\LL\leq 1 \}$ (or $\zeta$ in the previous sections). Indeed, the elements of ${\cal O_\LL}$ form a local ring and are of the form:
$$c_0+\sum_{q_j>0}c_{q_j} t^{q_j} .$$
This ring has a maximal ideal ${\frak M}:=\{  z\in\LL~|~|z|_\LL< 1 \}$ corresponding to the elements whose series has constant term 0.  Hence there is a natural field to consider called the residue field ${\cal O_\LL/ \frak M}$ and the natural projection called the reduction $\rho:\cal B_G\to \cal O_\LL/ \frak M\approx \C$ defined by 
$$c_0+\sum_{q_j>0}c_{j} t^{q_j} \longmapsto c_0.$$

Note that this makes sense even if the series cannot be identified to a holomorphic map.

\medskip
\noindent{\textbf{Dynamics.}}

We recalled that for $P_t:=X^d+a_1(t).X^{d-1}+\ldots+a_{d}(t)\in \LL[X]$ with $\forall i, a_i(t)\in\LL$, we can prove that 
$$P(B(0,r))=B(a_d(t),\underset{i\in[1,d]}{max}(|a_i(t)|_\LL . r^{d-i})).$$
In fact there is a well defined dynamics of $P_t$ on $\P^1 _{Berk}$t:
 if $|.|_x$ is a multiplicative semi-norm on $\LL[X]$ then $P(|.|_x)$ is the semi-norm that maps $Q \in\LL[X]$ to $|Q \circ P |_x$.

Note that proving that rational maps with coefficients in $\LL$ define a dynamics on $\P^1_{Berk}$ remains complicated and the interested reader can refer to \cite{Juan}.


\end{document}